\documentclass[12pt]{amsart}
\usepackage{latexsym, amsmath, amssymb, amsthm}
\usepackage{epsfig}
\usepackage{amscd}
\usepackage{latexsym}
\usepackage{cite}

\newtheorem{theorem}{Theorem}

\newtheorem{lemma}[theorem]{Lemma}
\newtheorem{corollary}[theorem]{Corollary}

\def\scaleddraw #1 by #2 (#3 scaled #4){{
  \dimen0=#1 \dimen1=#2
  \divide\dimen0 by 1000 \multiply\dimen0 by #4
  \divide\dimen1 by 1000 \multiply\dimen1 by #4
  \draw \dimen0 by \dimen1 (#3 scaled #4)}
  }

\newcommand{\C}[1]{\mathcal #1}
\newcommand{\B}[1]{\mathbb #1}

\begin{document}

\author{Dongseok Kim}
\address{Department of Mathematics \\ Kyungpook National University \\ Taegu, 702-201 Korea}
\email{dongseok@knu.ac.kr}
\thanks{}

\author{Young Soo Kwon}
\address{Department of Mathematics \\Yeungnam University \\Kyongsan, 712-749, Korea}
\email{ysookwon@yu.ac.kr}
\thanks{}

\author{Jaeun Lee}
\address{Department of Mathematics \\Yeungnam University \\Kyongsan, 712-749, Korea}
\email{julee@yu.ac.kr}
\thanks{This research was supported by the Yeungnam University research grants in 2007.}


\title[The weighted complexity and the determinant functions of graphs]{The weighted complexity and the determinant functions of graphs}

\begin{abstract}
The complexity of a graph can be obtained as a derivative of a variation of the zeta function [\textit{J. Combin. Theory Ser. B}, 74 (1998), pp. 408--410] or
a partial derivative of its generalized characteristic polynomial evaluated at a point [arXiv:0704.1431[math.CO]].
A similar result for the weighted complexity of weighted graphs was found using a determinant function [\textit{J. Combin. Theory Ser. B}, 89 (2003), pp. 17--26].
In this paper, we consider the determinant function of two variables and discover a condition that
the weighted complexity of a weighted graph is a partial derivative of the determinant function
evaluated at a point. Consequently, we simply obtain the previous results and disclose
a new formula for the Bartholdi zeta function.
We also consider a new weighted complexity, for which the weights of spanning trees are taken as the sum of weights of edges in the tree,
and find a similar formula for this new weighted complexity. As an application, we compute
the weighted complexities of the product of the complete graphs.
\end{abstract}

\maketitle

Let $G$ be a finite simple graph with vertex set $V(G)$,
edge set $E(G)$. Let $\nu_G$ and $\varepsilon_G$ denote the
number of vertices and edges of $G$, respectively. Let $\C{A}(G)$ and $\C{D}(G)$ be the adjacency matrix and degree matrix of
$G$, respectively. Then the \emph{admittance matrix} or
\emph{Laplacian matrix} $\C{L}(G)$ of $G$ is $\C{D}(G)-\C{A}(G)$.
For other general terms, we refer to ~\cite{CDS}.

One of classical problems in graph theory is to find the \emph{complexity} of $G$, $\kappa(G)$,
the number spanning trees in a graph $G$~\cite{Chow, HK}.
The celebrated Kirchhoff's matrix tree theorem finds that
$\kappa(G)$ is any cofactor of the admittance matrix (or Laplacian matrix) of $G$ which is a generalization of
Cayley's formula which provides
$\kappa(K_n)$ of the complete graph $K_n$ on $n$ vertices.
On the other hand, the polynomial invariants of graphs have played a key role in the study of graphs.
For instance, the chromatic polynomial $p_G(\lambda)$, introduced by Birkhoff, is a very
important invariant of $G$ that counts the number of $\lambda$-colorings of $G$~\cite{Bir}.
A generalization of the chromatic polynomial is the Tutte polynomial $T_G(x,y)$
of a graph $G$~\cite{Tutte1, Tutte2}, most easily defined as $$T_G(x,y)=R_G(x-1,y-1),$$
where $R_G(x-1,y-1)$ is the Whitney's rank generating function~\cite{whitney} and one can see that
 $\kappa(G)= T_G(1,1)$.
There are a few more bridges between the complexity and the polynomial invariants of graphs~\cite{KKL, MV, No}.
In~\cite{No}, Northshield found that $$f'_G(1) = 2(\varepsilon_G-\nu_G)\kappa(G),$$ where
$f_G(u) = \mathrm{det} [I -
u~\C{A}(G) + u^2~(\mathcal{D}(G)- I)]$.
In~\cite{KKL}, a similar results was shown for the generalized characteristic polynomials introduced by Cvetkovic and et al.~\cite{CDS}.

A {\em weighted graph} is a pair
$G_\omega = (G,\;\omega)$,
where $\omega : E(G) \to R$ is a function
on the set $E(G)$ of edges in $G$ and $R$ is a commutative ring with identity. We call $G$ the {\em underlying
graph} of $G_\omega$ and $\omega$ the {\em weight function} of $G_\omega$.
Given any weighted graph $G_\omega$,
the adjacency matrix $\C{A}(G_\omega) = (w_{ij})$ of $G_\omega$ is
the square matrix of order $\nu_G$ defined by
  $$ w_{ij}
     = \left\{ \begin{array}{cl}
         \omega(e)	& \quad \mbox{ if $e=\{ v_i,v_j\}\in E(G),$} \vspace{1mm} \\
       0  & \quad \mbox{ otherwise.} \vspace{2mm}
       \end{array} \right. $$
Notice that the adjacent matrix $\C{A}(G_\omega)$ of $G_\omega$ is symmetric.
The incidence matrix $\C{I}(G_\omega)=(i_{hk})$ of $G_\omega$, with respect to a given orientation, is defined by
$$ i_{hk}
     = \left\{ \begin{array}{rl}
         \omega(e_k)	& \quad \mbox{ if $ v_h$ is the positive end of $e_k$,} \vspace{1mm} \\
         -\omega(e_k)	& \quad \mbox{ if $ v_h$ is the negative end of $e_k$,} \vspace{1mm} \\
       0\,\,\,  & \quad \mbox{ otherwise.} \vspace{2mm}
       \end{array} \right. $$
The \emph{degree matrix}  $\mathcal{D}(G_{\omega})$ of $G_{\omega}$ is
the diagonal matrix whose $(i,i)$-th entry is $\omega_i^{G_\omega}$, the sum of the weights of edges adjacent to
$v_i$ in $G$ for each $1\le i\le \nu_G$. The \emph{admittance matrix} or \emph{Laplacian matrix}
$\C{L}(G_\omega)$ of $G_\omega$ is $\C{D}(G_{\omega})-\C{A}(G_{\omega})$.
Notice that every unweighted graph $G$ can be considered as the weighted graph whose weight function assigns 1 to each edge of $G$ and that
$\C{L}(G_\omega)=\C{D}(G_{\omega})-\C{A}(G_{\omega})=\C{I}(G_\omega)\C{I}(G)^t$,
 where $A^t$ is the transpose of the matrix $A$.

Mizuno and Sato~\cite{MS} considered the weighted complexity, and
generalized Northshield's result by showing
$$F'_{G_\omega}(1) = 2(\omega(G)-\nu_G)\kappa(G_{\omega}),$$ where
$F_{G_\omega}(u)  = \mathrm{det} [I -
u~\C{A}(G_{\omega}) + u^2~(\mathcal{D}(G_{\omega})- I)]$ and $\displaystyle\omega(S)=\sum_{e\in E(S)}\omega(e)$
for any subgraph $S$ of $G$.

Instead of considering these determinant functions individually, we start from the following general determinant function,
$$\Phi_{G_\omega}(\lambda, \mu)=\mathrm{det} [f(\lambda,\mu)I +
g(\lambda,\mu)\mathcal{D}(G_\omega) + h(\lambda,\mu)\C{A}(G_\omega)],$$
then find a condition that one can obtain
a generalization of matrix tree theorem.
Now, we are set to provide the main result as follows.

\begin{theorem} \label{mainthm}
Let $G_\omega$ be a finite weighted graph with $\nu_G$ vertices and $\epsilon_G$ edges whose weights on edges are complex numbers. Let $f(\lambda,\mu)$, $g(\lambda,\mu)$, and $h(\lambda,\mu)$
be partial differentiable functions
such that $f(\alpha,\beta)=0$ and $g(\alpha,\beta)+h(\alpha,\beta)=0$ for some $\alpha$ and $\beta$.
Then
$$\frac{\partial \Phi_{G_\omega}}{\partial \lambda}(\alpha,\beta)=g(\alpha,\beta)^{\nu_G-1}\left[f_\lambda(\alpha,\beta)\nu_G + (g_\lambda(\alpha,\beta)+h_\lambda(\alpha,\beta))2\omega(G)\right] \kappa(G_\omega),$$
and $$\frac{\partial \Phi_{G_\omega}}{\partial \mu}(\alpha,\beta)=g(\alpha,\beta)^{\nu_G-1}\left[f_\mu(\alpha,\beta)\nu_G + (g_\mu(\alpha,\beta)+h_\mu(\alpha,\beta))2\omega(G)\right] \kappa(G_\omega).$$
\end{theorem}

For Theorem~\ref{mainthm}, the definition of the weighted complexity $\kappa(G_\omega)$, found in Lemma~\ref{basiclem},
use the weight of a spanning tree as the product of weights on edges in the tree. Although it fits well with many occasions, it is much more
natural to consider the weight of a spanning tree as the sum of weights on edges in the tree. We call
this new weighted complexity, the sigma weighted complexity, $\kappa_\sigma(G_\omega)$. Then, we find Theorem~\ref{weightsumthm} as a counterpart of Theorem~\ref{mainthm} for this complexity, $\kappa_\sigma(G_\omega)$.

The outline of this paper is as follows. In section~\ref{main}, we first prove a couple of lemmas which
show the weighted complexity of weighted graphs is any cofactor of the Laplacian matrix
$\C{L}(G_\omega)$ of $G_\omega$. Then we provide a proof of Theorem~\ref{mainthm}.
We also consider a new weighted complexity for which the weights of spanning trees are taken as the sum of weights of edges in the tree,
and obtain a similar formula for this new weighted complexity. We also explain how previous results
can be obtained from our consummation. We also provide a new sequel for the Bartholdi zeta function.
Finally, we compute the weighted complexities of the product of the complete graphs in section~\ref{conseq}.

\section{Main results}\label{main}

Even though the matrix tree theorem of an unweighted graph $G$ finds that any cofactor of the Laplacian matrix of $G$ are the same~\cite{Ha},
it was previously known that all principal cofactors of the Laplacian matrix $\C{L}(G_\omega)$ of $G_\omega$ are the same~\cite{Cha,Or}.
These common values were defined
as the \emph{weighted complexity} $\kappa(G_\omega)$ of a weighted graph $G$~\cite{MS}. In
the following Lemma~\ref{basiclem}, we extend this definition that any cofactor of the Laplacian matrix
$\C{L}(G_\omega)$ of $G_\omega$ is the weighted complexity $\kappa(G_\omega)$, $i. e.,$ not only the principal cofactors
but also any cofactor of the Laplacian matrix $\C{L}(G_\omega)$ of $G_\omega$ are the same.
In the proof of the main theorem, the ring $R$ is the polynomial ring over real numbers,
but the following lemmas can be proven for a commutative ring with identity.

 \begin{lemma}\label{treelem}
Let $R$ be a commutative ring with identity and let $\omega:E(G)\to R$ be a weight function of a graph $G$.
Let $U$ be a subset of $E(G)$ having $\nu_G-1$ edges and let $\langle U \rangle$ be the spanning subgraph of $G$ induced by $U$.
Let $\C{I}(\langle U \rangle_\omega)_i$ be the matrix obtained by removing
$i$-th row of $\C{I}(\langle U \rangle_\omega)$. Then, for each $i=1,2,\ldots,\nu_G$,
$$\det(\C{I}(\langle U \rangle_\omega)_i)=(-1)^{i-1} \det(\C{I}(\langle U \rangle_\omega)_1)=(-1)^{i-1}\left(\prod_{e\in U}\omega(e)\right)\det(\C{I}(\langle U \rangle)_i),$$
where $\C{I}(\langle U \rangle)$ is the incidence matrix of the underlying tree $\langle U \rangle$ of $\langle U \rangle_\omega$. In particular, for each $i=1,2,\ldots,\nu_G$,
$$\det(\C{I}(\langle U \rangle)_i)=(-1)^{i-1} \det(\C{I}(\langle U \rangle)_1).$$
 \end{lemma}

\begin{proof}
For convenience, let $\C{I}(\langle U \rangle_\omega)=(\mathbf{r}_1,\mathbf{r}_2,\ldots, \mathbf{r}_{\nu_G})^t$.
Then $\mathbf{r}_1+\mathbf{r}_2+\ldots+\mathbf{r}_{\nu_G}=\mathbf{0}$. From this fact and properties of the determinant function, we can have that
$$\begin{array}{lcl}
&  & \det(\C{I}(\langle U \rangle_\omega)_i)  \\
& = & \det \left(\mathbf{r}_1, \mathbf{r}_2, \ldots, \mathbf{r}_{i-1}, \mathbf{r}_{i+1}, \ldots, \mathbf{r}_{\nu_T}\right)^t\\[1ex]
& = & (-1)\det \left( -\mathbf{r}_1-\mathbf{r}_2-\cdots-\mathbf{r}_{i-1}-\mathbf{r}_{i+1} -\cdots -\mathbf{r}_{\nu_T}, \mathbf{r}_2, \ldots, \mathbf{r}_{i-1},\mathbf{r}_{i+1}, \ldots, \mathbf{r}_{\nu_T}\right)^t\\[1ex]
& = &(-1)\det \left( \mathbf{r}_i, \mathbf{r}_2, \ldots, \mathbf{r}_{i-1}, \mathbf{r}_{i+1}, \ldots, \mathbf{r}_{\nu_T} \right)^t\\[1ex]
& = &(-1)^{i-1}\det \left(\mathbf{r}_2 , \ldots , \mathbf{r}_{i-1},\mathbf{r}_i, \mathbf{r}_{i+1}, \ldots, \mathbf{r}_{\nu_T} \right)^t\\[1ex]
& = &(-1)^{i-1} \det(\C{I}(\langle U \rangle_\omega)_1).
\end{array}
$$
Since, for each edge $e$ in $U$, $\omega(e)$ is a common factor of the column of $\C{I}(T_\omega)_1$ corresponding to the edge $e$,
we have
$$\det(\C{I}(\langle U \rangle_\omega)_1)=\left(\prod_{e\in E(T)}\omega(e)\right) \det(\C{I}(\langle U \rangle)_1).$$ It completes the proof.
\end{proof}

 \begin{lemma}\label{basiclem}
 Let $R$ be a commutative ring with identity and let $\omega:E(G)\to R$ be a weight function of a graph $G$.
 Then $$ \C{L}(G_{\omega})_{ij}=\sum_{T\in \C{T}(G)}\left(\prod_{e\in E(T)} \omega(e)\right),$$
for each $1\le i,j \le \nu_G$, where $\C{T}(G)$ is the set of all spanning trees in $G$ and $A_{ij}$ is the $ij$-cofactor of a matrix $A$.
 \end{lemma}

\begin{proof}
 Let $\C{I}(G_\omega)_i$ be the matrix obtained by removing
$i$-th row of $\C{I}(G_\omega)$. Then $\C{L}(G_\omega)_{ij}=(-1)^{i+j}\det(\C{I}(G_\omega)_i\,\C{I}(G)_j^t)$.
By applying Binet-Cauchy theorem and Lemma~\ref{treelem}, we can see
that
$$\begin{array}{lcl}
\det(\C{I}(G_\omega)_i\,(\C{I}(G)_j)^t) &= & \displaystyle \sum_{|U|=\nu_G-1}\det([\C{I}(G_\omega)_i]_U) \det(([\C{I}(G)_j]_U)^t)\\
&=& \displaystyle \sum_{|U|=\nu_G-1} (-1)^{i+j}\left(\prod_{e\in U} \omega(e)\right)\det([\C{I}(G)_1]_U)^2,
\end{array}$$
where $[\C{I}(G)_i]_U$ is the square submatrix of $\C{I}(G)_i$ whose $\nu_G-1$ columns corresponding to the edges in a subset $U$ of $E(G)$.
It is known that   $\det([\C{I}(G)_i]_U)\not=0$
if and only if the  subgraph $\langle U \rangle$ induced by $U$ is a spanning tree of $G$. Moreover, if $\langle U \rangle$ is a tree, then  $\det([\C{I}(G)_i]_U)=\pm 1$. (For example, see~\cite[Propositions 5.3 and 5.4]{Bi}). Form this, it can be shown that
$$
\det(\C{I}(G_\omega)_i\,(\C{I}(G)_j)^t)
=  \sum_{T\in \C{T}(G)} (-1)^{i+j}\left(\prod_{e\in E(T)} \omega(e)\right).$$
It completes the proof.
\end{proof}

Using Lemma~\ref{basiclem}, one can define the weighted complexity $\kappa(G_\omega)$ of a weighted graph $G_\omega$ by
$$\kappa(G_{\omega})\equiv\C{L}(G_{\omega})_{ij}.$$

Now we are set to proceed the proof of Theorem~\ref{mainthm}.
For convenience, let
$$\Phi_{G_\omega}(\lambda, \mu)=\det [{\mathbf c}_1(\lambda,\mu),{\mathbf c}_2(\lambda,\mu), \ldots, {\mathbf c}_{\nu_G}(\lambda,\mu)].$$
Then
$$\frac{\partial \Phi_{G_\omega}}{\partial \lambda}(\lambda, \mu)=\sum_{i=1}^{\nu_G}\det[{\mathbf c}_1(\lambda,\mu),{\mathbf c}_2(\lambda,\mu),\ldots, {\mathbf c}_{i-1}(\lambda,\mu), ({\mathbf c}_i)_\lambda(\lambda,\mu),{\mathbf c}_{i+1}(\lambda,\mu), \ldots, {\mathbf c}_{\nu_G}(\lambda,\mu)],$$
where $$\begin{array}{l}
({\mathbf c}_i)_\lambda(\lambda,\mu)\\
\hspace{.1cm}=[h_\lambda(\lambda,\mu) w_{1\,i},\ldots,h_\lambda(\lambda,\mu) w_{i-1\, i},
f_\lambda(\lambda,\mu)+ g_\lambda(\lambda,\mu)\omega_i^G, h_\lambda(\lambda,\mu)w_{i+1\, i},\ldots, h_\lambda(\lambda,\mu) w_{\nu_G\, i}]^t.
\end{array}$$
Since $f(\alpha,\beta)=0$ and $g(\alpha,\beta)+h(\alpha,\beta)=0$ for some $\alpha$ and $\beta$, we can see that
the expansion of the determinant $$\det[{\mathbf c}_1(\alpha,\beta),{\mathbf c}_2(\alpha,\beta),\ldots, {\mathbf c}_{i-1}(\alpha,\beta), ({\mathbf c}_i)_\lambda(\alpha,\beta),{\mathbf c}_{i+1}(\alpha,\beta), \ldots, {\mathbf c}_{\nu_G}(\alpha,\beta)]$$
with respect to the $i$-th column is
$$
g(\alpha,\beta)^{\nu_G-1}\left\{\left(f_\lambda(\alpha,\beta)+
g_\lambda(\alpha,\beta)\omega_i^G\right) \C{L}(G_\omega)_{ii}+  h_\lambda(\alpha,\beta)\sum_{k=1,k\not=i}^{\nu_G}\omega_{ki}\C{L}(G_\omega)_{ki}\right\}.
$$
Since  $\C{L}(G_\omega)_{ij}=\kappa(G_\omega)$ for each $1\le i,j\le \nu_G$, we have
$$\begin{array}{ccl} \displaystyle
\frac{\partial \Phi_{G_\omega}}{\partial \lambda}(\alpha, \beta)& = &\displaystyle  \sum_{i=1}^{\nu_G}\det[{\mathbf c}_1(\alpha, \beta),{\mathbf c}_2(\alpha, \beta),\ldots,  ({\mathbf c}_i)_\lambda(\alpha, \beta), \ldots, {\mathbf c}_{\nu_G}(\alpha, \beta)]\\[3ex]
 & =& \displaystyle \sum_{i=1}^{\nu_G} g(\alpha,\beta)^{\nu_G-1}\left\{f_\lambda(\alpha,\beta) + (g_\lambda(\alpha,\beta)+h_\lambda(\alpha,\beta))w_i^G\right\} \kappa(G_\omega)
 \\[3ex]
 & =& \displaystyle  g(\alpha,\beta)^{\nu_G-1}\left\{f_\lambda(\alpha,\beta)\nu_G + (g_\lambda(\alpha,\beta)+h_\lambda(\alpha,\beta))2\omega(G)\right\} \kappa(G_\omega).\\[1ex]
\end{array}$$
Similarly, we can have the second equation. It completes the proof. \qed
\medskip

Next, we will obtain another key theorem for which the weight of a spanning tree $T$ is defined by the
sum of weights of edges in $T$, different from that of $\kappa(G_\omega)$.
For a weighted graph $G_\omega$, the \emph{sigma weighted complexity}, denoted by $\kappa_\sigma(G_\omega)$,
is the sum of all weights in the edges of spanning trees in $G$, that is,
$$ \kappa_\sigma(G_\omega)=\sum_{T\in \C{T}(G)}\left(\sum_{e\in T} \omega(e)\right)=\sum_{T\in \C{T}(G)}\omega(T).$$
Then, for any constant weight function $\omega=c$, it is clear that $\kappa_\sigma(G_\omega)=c(\nu_G-1)\kappa(G)$. In particular,
$\kappa_\sigma(G)=(\nu_G-1)\kappa(G)$ for any graph $G$. For a weighted graph $G_\omega$ with $\omega:E(G)\to \B{C}$, we define a new weight function
 $\omega_x:E(G)\to \B{C}[x]$  by $\omega_x(e)=x^{\omega(e)}$. Then $\kappa(G_{\omega_x})'(1)=\kappa_\sigma(G_\omega)$. Now,
by using a method similar to the proof of Theorem~\ref{mainthm}, we have the following theorem.

\begin{theorem} \label{weightsumthm}
Let $G_\omega$ be a finite weighted graph with $\nu_G$ vertices and $\epsilon_G$ edges whose weights on edges are complex numbers. Let $f(\lambda,\mu)$, $g(\lambda,\mu)$, and $h(\lambda,\mu)$
 be partial differentiable functions
such that $f(\alpha,\beta)=0$ and $g(\alpha,\beta)+h(\alpha,\beta)=0$ for some $\alpha$ and $\beta$.
Then
\begin{align*}\displaystyle\frac{\partial^2 \Phi_{G_{\omega_x}}}{\partial x \partial \lambda}(\alpha,\beta,1)&=g(\alpha,\beta)^{\nu_G-1} \left[g_\lambda(\alpha,\beta)+h_\lambda(\alpha,\beta)\right]2\omega(G)\kappa(G) \\
\displaystyle &+ g(\alpha,\beta)^{\nu_G-1}\left[f_\lambda(\alpha,\beta)\nu_G + (g_\lambda(\alpha,\beta)+h_\lambda(\alpha,\beta))2\epsilon_G\right]\kappa_\sigma(G_\omega),\end{align*}
and \begin{align*}\displaystyle\frac{\partial^2 \Phi_{G_{\omega_x}}}{\partial x \partial \mu}(\alpha,\beta,1)&=g(\alpha,\beta)^{\nu_G-1} \left[g_\mu(\alpha,\beta)+h_\mu(\alpha,\beta)\right]2\omega(G)\kappa(G) \\
\displaystyle &+ g(\alpha,\beta)^{\nu_G-1}\left[f_\mu(\alpha,\beta)\nu_G + (g_\mu(\alpha,\beta)+h_\mu(\alpha,\beta))2\epsilon_G\right]\kappa_\sigma(G_\omega).\end{align*}
\end{theorem}

We will show how the previous results can be obtained from Theorem~\ref{mainthm} and~\ref{weightsumthm}.
 For $F_{G_\omega}(u)$, 
one can choose $f(\lambda,\mu)=1-\lambda^2$, $g(\lambda,\mu)=\lambda^2$, $h(\lambda,\mu)=-\lambda$
and $(\alpha, \beta)=(1,0)$. Consequently, we find the following corollaries from Theorem~\ref{mainthm}.

\begin{corollary} [\cite{No}] \label{zetacoro} Let $f_G(u)$ be a variation of the zeta function defined as $$f_G(u)= \mathrm{det} [I -
u~\C{A}(G) + u^2~(\mathcal{D}(G)- I)].$$ Then,
$$f_G'(1)=2(\varepsilon_G-\nu_G)\kappa(G).$$
\end{corollary}

\begin{corollary}[\cite{MS}]  \label{weightzetacoro} Let $F_{G_\omega}(u)$ be a variation of the zeta function of weighted graph
$G_\omega$ defined as $$F_{G_\omega}(u)= \mathrm{det} [I -
u~\C{A}(G_\omega) + u^2~(\mathcal{D}(G_\omega)- I)].$$ Then,
$$F_{G_\omega}'(1)=2(\omega(G)-\nu_G)\kappa(G_\omega).$$
\end{corollary}

By setting $f(\lambda,\mu)=\lambda$, $g(\lambda,\mu)=\mu$, $h(\lambda,\mu)=-1$
and $(\alpha, \beta)$$=(0,1)$, we find the following corollary.

\begin{corollary} [\cite{KKL}]  \label{gencharcoro}
Let $\C{F}_G(\lambda, \mu)$ be the generalized characteristic polynomial defined as
$$\C{F}_G(\lambda, \mu) = \mathrm{det} [\lambda I -(\C{A}(G) - \mu \mathcal{D}(G))].$$ Then,
$$\frac{\partial \C{F}_G}{\partial \mu}(0,1)=2\varepsilon_G\,\kappa(G).$$
\end{corollary}

\begin{theorem}  \label{Barthcoro}
Let $B_G(t,u)$ be a variation of the Bartholdi zeta function~\cite{Bar} of $G$ defined as
$$B_G(t,u)=\det\left[I-\C{A}(G)u+(1-t)(\C{D}(G)-(1-t)I)u^2\right].$$  Then the complexity $\kappa(G)$ of
$G$ can be obtained as follows,
$$\frac{\partial B_G}{\partial t}(1,0)=2(\nu_G-\varepsilon_G)\kappa(G)\quad\mbox{\rm and}
\quad \frac{\partial B_G}{\partial \mu}(1,0)=2(\varepsilon_G-\nu_G)\kappa(G).$$
\end{theorem}

\begin{proof}
By setting $f(t,u)=(1-(1-t)^2u^2)$, $g(t,u)=(1-t)u^2$, $h(t,u)= - u$
and $(\alpha, \beta)=(1,0)$. Note that $B_G(0,u)=f_G(u)$.
\end{proof}

\begin{corollary} \label{charcoro}
Let $\sigma_{G_\omega}(\mu)=\det\left[\mu I-(\C{D}(G_\omega)-\C{A}(G_\omega))\right]$  be the characteristic function of
the Laplacian matrix. Then
$$\sigma_{G_\omega}'(0)=(-1)^{\nu_{G}-1}\,\nu_G\,\kappa(G_\omega).$$
\end{corollary}

\section{The weighted complexities of the product of the complete graphs}\label{conseq}

To demonstrate Theorem~\ref{mainthm} and Theorem~\ref{weightsumthm}, we consider
the product of the complete graphs $K_{m_1}\times
K_{m_2}\times \ldots \times K_{m_n}\equiv G(m_1, m_2, \ldots, m_n)$ whose vertices are all $n$-tuples of numbers
$a_i$ where $a_i\in \{1, 2, \ldots, m_i\}$ and $i= 1, 2, \ldots, n$
and two vertices $\mathbf{a}=(a_1, a_2, \ldots, a_n)$
and $\mathbf{b}=(b_1, b_2, \ldots, b_n)$ are adjacent if and only if ${\mathbf a}$ and ${\mathbf b}$ differ in exactly
one coordinate. We define a weight function $\omega: E(G(m_1, m_2, \ldots, m_n))\rightarrow \{\omega_1, \omega_2, \ldots, \omega_n\}$ by
$\omega(\{{\mathbf a}, {\mathbf b}\})=\omega_i$ if ${\mathbf a}$ and ${\mathbf b}$ differ in the $i$-th coordinate.
Then
$$\det\left(\lambda I-\C{L}({G(m_1, m_2, \ldots, m_n)}_\omega)\right)=\lambda
\prod_{\emptyset \not=S\subset \{1,2,\ldots, n\}}\left(\lambda-\sum_{s\in S}m_s \omega_s\right)^{\prod_{s\in S}(m_s-1)}.$$
By applying Theorem~\ref{mainthm} and the fact that $ \displaystyle(-1)^{\sum_{\emptyset \not=S\subset \{1,2,\ldots, n\}}\prod_{s\in S}(m_s-1)}=(-1)^{m_1m_2\ldots m_n-1}=(-1)^{\nu_{G(m_1, m_2, \ldots, m_n)}-1}$, we have
$$\kappa({G(m_1, m_2, \ldots, m_n)}_\omega) \left(\prod_{i=1}^n m_i\right)= \prod_{\emptyset \not=S\subset \{1,2,\ldots, n\}}\left(\sum_{s\in S}m_sw_s\right)^{\prod_{s\in S}(m_s-1)},$$ and
$$
\kappa({G(m_1, m_2, \ldots, m_n)}) \left(\prod_{i=1}^n m_i\right)=\prod_{\emptyset \not=S\subset \{1,2,\ldots, n\}}\left(\sum_{s\in S}m_s\right)^{\prod_{s\in S}(m_s-1)}.$$
In particular, if $m_s=m$ for all $s=1,2, \ldots, n$,
\begin{eqnarray*}\kappa({G(m, m, \ldots, m)}_\omega) m^n&=& \prod_{\emptyset \not=S\subset \{1,2,\ldots, n\}}m^{(m-1)^{|S|}}\left(\sum_{s\in S}\omega_s\right)^{(m-1)^{|S|}}\\
&=& \left( \prod_{k=1}^n m^{{n\choose k}(m-1)^k}\right) \left[\prod_{\emptyset \not=S\subset \{1,2,\ldots, n\}}\left(\sum_{s\in S}\omega_s\right)^{(m-1)^{|S|}}\right]. \end{eqnarray*}
Thus, we find
$$\kappa({G(m, m, \ldots, m)}_\omega) = m^{m^n-n-1}\left[ \prod_{\emptyset \not=S\subset \{1,2,\ldots, n\}}\left(\sum_{s\in S}\omega_s\right)^{(m-1)^{|S|}}\right],$$
and $$\kappa({G(m, m, \ldots, m)}) = m^{m^n-n-1}\left( \prod_{k=1}^n k^{{n\choose k}(m-1)^k}\right).$$
For $m=2$, $G(2, 2, \ldots, 2)$ is the $n$-dimensional hypercube $Q_n$, and its weighted complexity is
$$\kappa((Q_n)_\omega) = 2^{2^n-n-1} \prod_{\emptyset \not=S\subset \{1,2,\ldots, n\}}\left(\sum_{s\in S}\omega_s\right),$$
and its complexity is
$$\kappa(Q_n) = 2^{2^n-n-1} \left( \prod_{k=1}^n k^{n\choose k}\right).$$

Similarly, by using Theorem~\ref{weightsumthm} we have
$$\kappa_\sigma({G(m_1, m_2, \ldots, m_n)}_{\omega})\left(\prod_{i=1}^n m_i\right)=
\sum_{\overset{\emptyset \not=S,}{S\subset \{1,2,\ldots, n\}}}
\left(\prod_{\overset{\emptyset\neq T \not=S,}{T\subset \{1,2,\ldots, n\}}}\left(\sum_{t\in T}m_t\right)^{\prod_{t\in T}(m_t-1)}\right)\Omega(S),$$
where $$\Omega(S)=\left[\prod_{s\in S}(m_s-1)\left(\sum_{s\in S}m_s\right)^{-1+\prod_{s\in S} (m_s-1)}\left(\sum_{s\in S}m_s\omega_s\right)\right].$$
and the sigma weighted complexity of $G(m, m, \ldots, m)$ is
$$\begin{array}{lcl}
& &\kappa_\sigma({G(m, m, \ldots, m)}_\omega)\\
&=&\displaystyle m^{m^n-n-1} \sum_{\overset{\emptyset \not=S,}{S\subset \{1,2,\ldots, n\}}}
\left(\prod_{\overset{\emptyset\neq T \not=S,}{T\subset \{1,2,\ldots, n\}}} |T|^{(m-1)^{|T|}} \right)\left(\sum_{s\in S} \omega_s \right)(m-1)^{|S|}|S|^{(m-1)^{|S|}-1}\\[5ex]
&=& \displaystyle m^{m^n-n-1} \prod_{k=1}^n k^{{n\choose k}(m-1)^k}\left(\sum_{\emptyset \not=S\subset \{1,2,\ldots, n\}} (m-1)^{|S|}\frac{\sum_{s\in S} \omega_s}{|S|}\right)\\[5ex]
&=& \displaystyle m^{m^n-n-1} \prod_{k=1}^n k^{{n\choose k}(m-1)^k}\left(\sum_{k=1}^n (m-1)^{k}{n\choose k}\frac{\omega_1+\omega_2 + \ldots + \omega_n}{n}\right)\\[5ex]
&=& \displaystyle m^{m^n-n-1} \left(\prod_{k=1}^{n} k^{{n\choose k}(m-1)^k}\right)\frac{m^n-1}{n}(\omega_1+\omega_2+\ldots+\omega_n).\end{array}
$$

We observe that every spanning tree in ${Q_n}_\omega$  contains at least one edge of weight $\omega_i$  for each $i=1,2,\ldots,n$.
Let $\omega_1\le \omega_2\le \ldots \le \omega_n$. By applying Kruskal's algorithm  to ${Q_n}_\omega$, we can
find a minimum spanning tree whose edge set is $E_1\cup E_2\cup \cdots \cup E_n$, where
$E_i=\{\,\{(0,\ldots,0,0,*),(0,\ldots,0,1,*)\}\,|\, *\in \{0,1\}^{n-i}\,\}$ for each $i=1,2,\ldots,n$.
Since  $\omega(e)=\omega_i$ for all $e\in E_i$ and $|E_i|=2^{n-i}$ ($i=1,2,\ldots,n$),  we have
 $$\min\{ \kappa_\sigma(T)\,:\, \mbox{$T$ is a spanning tree of the weighted graph ${Q_n}_\omega$}\}=\sum_{i=1}^n 2^{n-i}w_i.$$

\end{document}